\documentclass{article}
\usepackage{hyperref}
\usepackage[utf8]{inputenc}
\usepackage{dsfont}
\usepackage{amsthm}
\usepackage{amsfonts}
\usepackage{amssymb}
\usepackage{amsmath}
\usepackage[english]{babel}
\theoremstyle{definition}
\newtheorem{conj}{Conjecture}[section]
\theoremstyle{definition}
\newtheorem{thm}[conj]{Theorem}
\theoremstyle{definition}
\newtheorem{cor}[conj]{Corollary}
\theoremstyle{definition}
\newtheorem{prop}[conj]{Proposition}
\theoremstyle{definition}

\theoremstyle{definition}
\newtheorem{definition}[conj]{Definition}
\theoremstyle{definition}
\newtheorem{rem}[conj]{Remark}
\theoremstyle{definition}

\usepackage{graphicx} 

\title{On $C$-Algebraic and $C$-Arithmetic Automorphic Representations }
\author{Alfio Fabio La Rosa}
\date{}

\begin{document}

\maketitle

\begin{abstract} The main purpose of this article is to begin an investigation based on the trace formula of the conjecture of K. Buzzard and T. Gee proposing that every $C$-algebraic automorphic representation is $C$-arithmetic. In the setting of anisotropic groups, we prove a criterion which shows that, under certain conditions, the trace formula can be used to establish that an automorphic representation is $C$-arithmetic. This criterion does not require the existence of a geometric model for the non-Archimedean part of the representation: it only depends on the possibility of isolating a finite family of automorphic representations containing the given one, and on a certain arithmetic condition on the Archimedean orbital integrals. As an application, we give a new proof that automorphic representations with a regular discrete series as Archimedean component are $C$-arithmetic, and we suggest a possible road to study automorphic representations with more general Archimedean type. In the last section, we show that the conjecture for a general reductive group can be reduced to the analogous statement for cuspidal $C$-algebraic automorphic representations.

\end{abstract}

\tableofcontents

\section{Introduction} 

Understanding the arithmetic of Hecke operators acting on automorphic representations is, by now, a classical problem. In their article \cite{BG}, K. Buzzard and T. Gee introduced the notion of $C$-algebraic automorphic representation: this is an Archimedean notion, in the sense that it depends only on the infinitesimal character of the Archimedean component of the automorphic representation. They conjectured that $C$-algebraic automorphic representations are $C$-arithmetic; that is: there exists a number field that contains all the unramified Hecke eigenvalues (see Section 2 for precise definitions). \\

To the best of our knowledge, the only automorphic representations which are known to be $C$-arithmetic are those whose non-Archimedean part contributes to the cohomology of a locally symmetric space or to the coherent cohomology of a Shimura variety. \\

Many $C$-algebraic automorphic representations do not admit such a geometric realisation. It is for this reason that we begin an investigation of the conjecture of K. Buzzard and T. Gee based on the trace formula. \\

It seems that the only attempt to employ the trace formula to study this kind of question has been made by A. Knightly and C. Li in \cite{KL}: they showed that the Hecke eigenvalues of automorphic representations of $\text{GL}_{2}$ whose Archimedean component is an integrable discrete series are algebraic integers. Their method requires an explicit computation of the geometric side of the trace formula and it does not establish the $C$-arithmetic property.\\

In the context of connected semisimple anisotropic algebraic groups defined over $\mathbb{Q}$, we establish the following criterion (see Corollary \ref{Main}, sections 2 and 3 for notation and for the relevant choices of Haar measures):

\begin{thm}  Let $G$ be a connected, semisimple, anisotropic algebraic group defined over $\mathbb{Q}$. Let $\pi_{0}=\pi_{0,\infty}\otimes \pi_{0}^{\infty}$ be an irreducible automorphic representation of $G(\mathbb{A}_{\mathbb{Q}})$. Assume the existence of a function $f\in C^{\infty}_{c}(G(\mathbb{R}))$ verifying the following conditions:
\begin{enumerate}
\item[(1)] There exists a finite family $\mathcal{C}$ of irreducible automorphic representations of $G(\mathbb{A}_{\mathbb{Q}})$ containing $\pi_{0}$ such that:
\begin{enumerate}
\item[(i)] If $\pi\in\mathcal{C}$, then $\pi^{\infty,K^{\infty}}\neq\{0\}$ and $\text{Tr}\pi_{\infty}(f)=1$.
\item[(ii)] If $\pi\notin \mathcal{C}$, then $\text{Tr}\pi_{\infty}(f)=0$.
\end{enumerate}
\item[(2)] There exists a number field $E$ such that, for every $\gamma\in \{G(\mathbb{Q})\}$, the Archimedean orbital integral $\mathcal{O}_{\gamma_{\infty}}(f)$ belongs to $E$. 
\end{enumerate}
Then there exists a number field $E'$ and a finite set of places, $S$, containing $\infty$ and all the places at which $\pi_{0}$ ramifies, such that for every $p\notin S$, and for every $h\in\mathcal{H}_{\mathbb{Q}}(G(\mathbb{Q}_{p}),K_{p})$, the eigenvalue of the operator $$\pi_{0,p}(h):\pi_{0,p}^{K_{p}}\longrightarrow \pi_{0,p}^{K_{p}}$$ is contained in $E'$; that is: the representation $\pi_{0}$ is $C$-arithmetic.
\end{thm}

The result above leads to a simple, non-cohomological proof that automorphic representations of connected semisimple anisotropic algebraic groups defined over $\mathbb{Q}$ whose Archimedean component is a regular discrete series are $C$-arithmetic (see Corollary \ref{Application}). \\

To treat more general automorphic representations, the requirement that the test function in the criterion be smooth and compactly supported is too restrictive. In view of the results in \cite{BPLZZ}, we should probably allow Schwartz functions. However, for us, the main point at this stage is making the case that the trace formula can be used to investigate arithmetic properties of automorphic representations. Since smooth compactly supported functions are enough for the application in this article, we decided to state our criterion in this form. \\

Concerning (1), the article \cite{BPLZZ} provides a very powerful method to isolate cuspidal automorphic representations and we are currently exploring the possibility of employing it in our context. One difficulty is that the multipliers constructed in \cite{BPLZZ} are global objects: roughly, it is needed to modify a given test function by introducing Hecke operators which are, a priori, not $\mathbb{Q}$-valued (see, for example, \cite{BPLZZ} Proposition 3.17, already in the case $M=G$). We cannot afford this freedom in our context: we can only use Archimedean multipliers.
For anisotropic groups, however, combining Lemma 2.18 and the `Archimedean part' of Proposition 3.17, it seems possible to satisfy (1) of the criterion. \\

Concerning (2), we don't feel confident enough to formulate a precise conjecture, but, at least in this simplified picture, it could indicate why a purely Archimedean condition ($C$-algebraicity) might constrain the behaviour of the non-Archimedean part of the automorphic representation. \\

The proof that the Hecke eigenvalues are algebraic (Proposition \ref{algtrace}) is inspired by the method in \cite{KL}. To establish the $C$-arithmetic property we exploit an observation of G. Wiese (Proposition \ref{Gabor}). \\ 

Having in mind the idea of employing the trace formula to treat non-cuspidal automorphic representations of more general reductive groups, in Section 4 we reduce the conjecture of K. Buzzard and T. Gee that $C$-algebraic automorphic representations are $C$-arithmetic to the analogous statement for cuspidal representations (Theorem \ref{redcusp}).

\begin{thm}\label{redcusp} Let $G$ be a connected, reductive group defined over a number field $F$. If every irreducible, $C$-algebraic, cuspidal automorphic representation of $G(\mathbb{A}_{F})$ is $C$-arithmetic, then every irreducible, $C$-algebraic automorphic representation of $G(\mathbb{A}_{F})$ is $C$-arithmetic, and every irreducible, $L$-algebraic automorphic representation of $G(\mathbb{A}_{F})$ is $L$-arithmetic. 
\end{thm}

\textbf{Acknowledgements:} I am grateful to J. R. Getz for encouraging me to write this article and for many insightful conversations. I wish to thank S-W. Shin and T. Gee for answering several questions. Finally, I wish to thank A-M. Aubert for her advice and G. Wiese, my supervisor, for all his support.

\section{Background}

Let us begin by recalling some notions introduced in \cite{BG}. Our treatment follows \cite{GH}, as well. The notions of $L$-algebraic and $L$-arithmetic automorphic representation are needed in Section 4 only.\\

Let $G$ be a connected reductive algebraic group defined over a number field $F$. Let $v$ be an Archimedean place of $F$, and $F_{v}$ be the completion of $F$ at $v$. We identify the algebraic closure $\overline{F}_{v}$ of $F_{v}$ with $\mathbb{C}$. Let $T_{v}$ be a maximal torus of $G_{\mathbb{C}}$ and $B_{v}$ a Borel subgroup  containing $T_{v}$. Let $\rho_{B_{v}}$ denote half of the sum of the positive roots corresponding to this choice, and write $X^{*}(T_{v})$ for the group of characters of $T_{v}$. Finally, let $\mathfrak{t}_{v}$ denote the Lie algebra of $T_{v}$ and write $\mathfrak{t}_{v}^{\mathbb{C}}:=\mathfrak{t}\otimes_{\mathbb{R}}\mathbb{C}$.

\begin{definition} An irreducible admissible representation $\pi$ of $G(F_{v})$ with infinitesimal character $\lambda_{\pi}$ is $C$-algebraic if $\lambda_{\pi}-\rho_{B_{v}}\in X^{*}(T_{v})$.  
\end{definition}

\begin{definition} An irreducible admissible representation $\pi$ of $G(F_{v})$ with infinitesimal character $\lambda_{\pi}$ is $L$-algebraic if $\lambda_{\pi}\in X^{*}(T_{v})$.  
\end{definition}

The reader familiar with these notions will realise that they are not formulated as in \cite{BG}. However, as it is explained in loc. cit. and in Lemma 12.8.1 in \cite{GH}, our formulation is equivalent to the one in \cite{BG}. Also, these notions are independent of the choice of $B_{v}$.\\

Let $\mathbb{A}_{F}$ denote the ad\`ele ring of $F$ and $\mathbb{A}^{\infty}_{F}$ (resp. $\mathbb{A}_{F,\infty}$) the ring of finite ad\`eles (resp. the Archimedean part of the ring of ad\`eles). Fix a compact open subgroup $K^{\infty}=\Pi_{v}K_{v}$ of $G(\mathbb{A}^{\infty}_{F})$ with $K_{v}$ a hyperspecial maximal compact subgroup equal to $G(\mathcal{O}_{v})$ for all $v$ at which $G(F_{v})$ is unramified. Let $K_{\infty}$ denote a maximal compact subgroup of $G(\mathbb{A}_{F,\infty})$ and form the compact subgroup \\ $K:=K_{\infty}\times K^{\infty}$ of $G(\mathbb{A}_{F})$. The local definitions above lead to the following global notions:

\begin{definition} An irreducible automorphic representation $\pi=\bigotimes_{v} \pi_{v}$ of $G(\mathbb{A}_{F})$ is $C$-algebraic if $\pi_{v}$ is $C$-algebraic for every Archimedean place $v$ of $F$. 
\end{definition}

\begin{definition} An irreducible automorphic representation $\pi=\bigotimes_{v} \pi_{v}$ of $G(\mathbb{A}_{F})$ is $L$-algebraic if $\pi_{v}$ is $L$-algebraic for every Archimedean place $v$ of $F$. 
\end{definition}

We refer the reader to \cite{BG} for a detailed discussion of the rationale for introducing these notions and for how they are related to other aspects of the Langlands Program. We proceed to introduce the purely non-Archimedean notions of $C$-arithmetic and $L$-arithmetic automorphic representations. \\

Let $v$ be a non-Archimedean place of $F$ at which $G(F_{v})$ is unramified. Let $\mathcal{H}_{\mathbb{C}}(G(F_{v}),K_{v})$ denote the $\mathbb{C}$-algebra of bi-$K_{v}$-invariant, compactly supported, complex-valued functions on $G(F_{v})$, and $\mathcal{H}_{\mathbb{Q}}(G(F_{v}),K_{v})$ denote the $\mathbb{Q}$-subalgebra of $\mathbb{Q}$-valued elements in $\mathcal{H}_{\mathbb{C}}(G(F_{v}),K_{v})$. It is well-known that if $\pi$ is a smooth irreducible admissible representation of $G(F_{v})$ with non-trivial $K_{v}$-invariant vectors, then $\mathcal{H}_{\mathbb{C}}(G(F_{v}),K_{v})$ acts by a character $$\mathcal{H}_{\mathbb{C}}(G(F_{v}),K_{v})\longrightarrow \mathbb{C}$$ on the $1$-dimensional space of $K_{v}$-invariants $\pi^{K_{v}}$. 

\begin{definition} Let $E$ be a number field. Let $v$ be a non-Archimedean place of $F$ at which $G(F_{v})$ is unramified. We say that a smooth irreducible admissible representation $\pi$ of $G(F_{v})$ is defined over $E$ if the image of the map $$\mathcal{H}_{\mathbb{Q}}(G(F_{v}),K_{v})\longrightarrow \mathbb{C}$$ induced by the character $\mathcal{H}_{\mathbb{C}}(G(F_{v}),K_{v})\longrightarrow \mathbb{C}$ is contained in $E$.
\end{definition}

\begin{definition} An irreducible automorphic representation $\pi=\bigotimes_{v} \pi_{v}$ of $G(\mathbb{A}_{F})$ is $C$-arithmetic if there exist a number field $E$ and a finite set of places, $S$, containing all the Archimedean places and all the places at which $\pi_{v}$ ramifies, such that $\pi_{v}$ is defined over $E$ for all $v\notin S$.
\end{definition}

It only remains to introduce the notion of $L$-arithmetic automorphic representation. \\

Again, let $v$ be a place of $F$ at which $G(F_{v})$ is unramified. Let $T_{v}$ be a maximal torus of $G(F_{v})$ and $B_{v}$ a Borel subgroup containing $T_{v}$. Recall that the Satake isomorphism identifies $\mathcal{H}_{\mathbb{C}}(G(F_{v}),K_{v})$ with $\mathbb{C}\left[X_{*}(T_{v,d})\right]^{W_{v,d}}$, where $X_{*}(T_{v,d})$ denotes the maximal split subtorus of $T_{v}$ and $W_{v,d}$ is the subgroup of the Weyl group of $G(F_{v})$ consisting of the elements leaving $T_{v,d}$ stable. Given a smooth irreducible admissible representation $\pi$ of $G(F_{v})$, we thus obtain a map $$\mathbb{C}\left[X_{*}(T_{v,d})\right]^{W_{v,d}}\longrightarrow \mathbb{C}$$ by pre-composing the character $\mathcal{H}_{\mathbb{C}}(G(F_{v}),K_{v})\longrightarrow \mathbb{C}$ with the Satake isomorphism.  

\begin{definition} Let $E$ be a number field. Let $v$ be a non-Archimedean place of $F$ at which $G(F_{v})$ is unramified and $\pi$ be a smooth irreducible admissible representation of $G(F_{v})$. We say that the Satake parameter of $\pi$ is defined over $E$ if the image of the map $$\mathbb{Q}\left[X_{*}(T_{v,d})\right]^{W_{v,d}}\longrightarrow \mathbb{C}$$ induced by the map $\mathbb{C}\left[X_{*}(T_{v,d})\right]^{W_{v,d}}\longrightarrow \mathbb{C}$ is contained in $E$.
\end{definition}

\begin{definition} An irreducible automorphic representation $\pi=\bigotimes_{v} \pi_{v}$ of $G(\mathbb{A}_{F})$ is $L$-arithmetic if there exist a number field $E$ and a finite set of places, $S$, containing all the Archimedean places and all the places at which $\pi_{v}$ ramifies, such that the Satake parameter of $\pi_{v}$ is defined over $E$ for all $v\notin S$. 
\end{definition}

As explained in \cite{BG}, the motivation for introducing two notions of arithmetic automorphic representations is the observation that $\mathcal{H}_{\mathbb{Q}}(G(F_{v}),K_{v})$ and $\mathbb{Q}\left[X_{*}(T_{v,d})\right]^{W_{v,d}}$ provide two $\mathbb{Q}$-structure of $\mathcal{H}_{\mathbb{C}}(G(F_{v}),K_{v})$ which, in general, do not coincide. In loc. cit., the following conjecture is proposed:

\begin{conj}\label{BGconj} Let $G$ be a connected reductive algebraic group defined over a number field $F$. An irreducible automorphic representation $\pi=\bigotimes_{v}\pi_{v}$ is $C$-algebraic (respectively, $L$-algebraic) if and only if it is $C$-arithmetic (respectively, $L$-arithmetic).
\end{conj}

\section{Employing the Trace Formula}

Let $G$ be a connected, semisimple, anisotropic algebraic group defined over $\mathbb{Q}$. Let $\mathbb{A}_{\mathbb{Q}}$ denote the adeles of $\mathbb{Q}$ and $\mathbb{A}^{\infty}_{\mathbb{Q}}$ the finite adeles of $\mathbb{Q}$. Let $K:=K_{\infty}\times K^{\infty}$, where $K_{\infty}$ is a maximal compact subgroup of $G(\mathbb{R})$ and $K^{\infty}$ a compact open subgroup of $G(\mathbb{A}_{\mathbb{Q}}^{\infty})$ such that $K_{p}=G(\mathbb{Z}_{p})$ for every $p$ at which $G$ is unramified. \\

We recall that the right regular representation $(R,L^{2}(\left[G\right]))$, where\\ $\left[G\right]:=G(\mathbb{Q})\backslash G(\mathbb{A}_{\mathbb{Q}})$, decomposes into a Hilbert direct sum $$L^{2}(\left[G\right])=\bigoplus_{\pi\in\mathcal{A}(G)}m(\pi)\pi.$$
Here, $\mathcal{A}(G)$ denotes the set of equivalence classes of irreducible automorphic representations of $G(\mathbb{A}_{\mathbb{Q}})$, the multiplicity $m(\pi)$ is finite for every $\pi\in \mathcal{A}(G)$.\\

In the following, unless otherwise stated, $\pi$ will always denote an element in $\mathcal{A}(G)$.\\

If $f=f_{\infty}\otimes f^{\infty}$ is an element in $C_{c}^{\infty}(G(\mathbb{A}_{\mathbb{Q}}))$, then the operator $R(f)$ is trace-class and we have the well known trace formula $$\sum_{\pi\in\mathcal{A}(G)}m(\pi)\text{Tr}\pi_{\infty}(f_{\infty})\text{Tr}\pi^{\infty}(f^{\infty})=\sum_{\gamma\in \{G(\mathbb{Q})\}}\text{vol}(G_{\gamma}(\mathbb{Q})\backslash G_{\gamma}(\mathbb{A}_{\mathbb{Q}}))\mathcal{O}_{\gamma}(f),$$
where $\{G(\mathbb{Q})\}$ denotes the set of conjugacy classes of $G(\mathbb{Q})$, the orbital integral $\mathcal{O}_{\gamma}(f)$ is defined as $$\mathcal{O}_{\gamma}(f):=\int_{G_{\gamma}(\mathbb{A}_{\mathbb{Q}})\backslash G(\mathbb{A}_{\mathbb{Q}})}f(x^{-1}\gamma x)\,dx,$$ and $G_{\gamma}(\mathbb{Q})$ (resp. $G_{\gamma}(\mathbb{A}_{\mathbb{Q}})$) denotes the stabiliser of $\gamma$ in $G(\mathbb{Q})$ (resp. in $G(\mathbb{A}_{\mathbb{Q}})$).\\

Proposition \ref{algtrace} and Corollary \ref{Main} below require Haar measures on $G(\mathbb{A}_{\mathbb{
Q}})$ and on the centralisers $G_{\gamma}(\mathbb{A}_{\mathbb{Q}})$ satisfying two conditions:

\begin{enumerate} 
\item[(M1)] For every $\gamma\in G(\mathbb{Q})$, the quantity $\text{vol}(G_{\gamma}(\mathbb{Q})\backslash G_{\gamma}(\mathbb{A}_{\mathbb{Q}}))$ is a rational number.
\item[(M2)] For every non-Archimedean place, the local Haar measure assigns a rational number to every open compact subset. For every non-Archimedean place $p$ at which $G(\mathbb{Q}_{p})$ is unramified, the local Haar measure $\mu_{p}$ on $G({\mathbb{Q}_{p}})$ assigns measure $1$ to $G(\mathbb{Z}_{p})$.
\end{enumerate}

Realising condition (M1) is non-trivial: in Corollary 3.3 below we will choose Gross' measure to fulfil it.\newline

For the first part of condition (M2) we argue as follows. By V.5.2 in \cite{Ren}, $G(\mathbb{Q}_{p})$ admits a maximal compact subgroup $K_{0}$ with the following property: there exists a neighbourhood basis of the identity, say $\mathcal{B}$, consisting of open compact subgroups and such that every $C\in \mathcal{B}$ is a normal subgroup of $K_{0}$. Normalising the Haar measure on $G(\mathbb{Q}_{p})$ so that $K_{0}$ has measure $1$, we obtain a Haar measure that assigns a rational number to every $C\in\mathcal{B}$ and, therefore, it assigns a rational number to every open compact subset. For $p$ such that $G(\mathbb{Q}_{p})$ is unramified, this reasoning shows that it is possible to normalise the local Haar measure so that it gives measure $1$ to $G(\mathbb{Z}_{p})$ and it assigns a rational number to every compact open subsets of $G(\mathbb{Q}_{p})$.  \\

We begin by proving that, under certain conditions on the automorphic representation we want to study, the trace formula can be used to establish the algebraicity of the eigenvalues of a $\mathbb{Q}$-valued Hecke operator. This is inspired by the method developed in \cite{KL}: it relies on the Newton-Girard identities.

\begin{prop}\label{algtrace} Let $G$ be a connected, semisimple, anisotropic algebraic group defined over $\mathbb{Q}$. Let $\pi_{0}=\pi_{0,\infty}\otimes \pi_{0}^{\infty}$ be an irreducible automorphic representation of $G(\mathbb{A}_{\mathbb{Q}})$. Assume the existence of a function $f\in C^{\infty}_{c}(G(\mathbb{R}))$ verifying the following conditions:
\begin{enumerate}
\item[(1)] There exists a finite family $\mathcal{C}$ of irreducible automorphic representations of $G(\mathbb{A}_{\mathbb{Q}})$ containing $\pi_{0}$ such that:
\begin{enumerate}
\item[(i)] If $\pi\in\mathcal{C}$, then $\pi^{\infty,K^{\infty}}\neq \{0\}$ and $\text{Tr}\pi_{\infty}(f)=1$.
\item[(ii)] If $\pi\notin \mathcal{C}$, then $\text{Tr}\pi_{\infty}(f)=0$.
\end{enumerate}
\item[(2)] For every $\gamma\in \{G(\mathbb{Q})\}$, the local orbital integral $\mathcal{O}_{\gamma_{\infty}}(f)$ is an algebraic number. 
\end{enumerate}
Then, for every $h\in \mathcal{H}_{\mathbb{Q}}(G(\mathbb{A}_{\mathbb{Q}}^{\infty}),K^{\infty})$, the eigenvalues of the operator $$\pi_{0}^{\infty}(h):\pi_{0}^{\infty,K^{\infty}}\longrightarrow \pi_{0}^{\infty,K^{\infty}}$$ are algebraic.
\end{prop}
\begin{proof} Let $f\in C^{\infty}_{c}(G(\mathbb{R}))$ be as in the statement and consider $h\in\mathcal{H}_{\mathbb{Q}}(G(\mathbb{A}_{\mathbb{Q}}^{\infty}),K^{\infty})$.
Applying the trace formula to compute the trace of $R(f\otimes h)$, we obtain
\begin{equation*}
\begin{split}
\text{Tr}\left[R(f\otimes h)\right]&=\sum_{\pi\in\mathcal{A}(G)}m(\pi)\text{Tr}\pi_{\infty}(f)\text{Tr}\pi^{\infty}(h)\\
&=\sum_{\pi\in\mathcal{C}}m(\pi)\text{Tr}\pi^{\infty}(h)
\end{split}    
\end{equation*}
on the spectral side, and therefore
\begin{equation*}
\begin{split}
\sum_{\pi\in\mathcal{C}}m(\pi)\text{Tr}\pi^{\infty}(h)&=\sum_{\gamma\in \{G(\mathbb{Q})\}}\textrm{vol}(G_{\gamma}(\mathbb{Q})\backslash
G_{\gamma}(\mathbb{A}))\mathcal{O}_{\gamma}(f\otimes h).
\end{split}
\end{equation*}
The spectral side is equal to the trace of the operator $$T:=\bigoplus_{\pi\in\mathcal{C}}m(\pi)\pi^{\infty}(h):\bigoplus_{\pi\in\mathcal{C}}m(\pi)\pi^{\infty,K^{\infty}}\longrightarrow \bigoplus_{\pi\in\mathcal{C}}m(\pi)\pi^{\infty,K^{\infty}}.$$
The sum on the geometric side is finite by Lemma 9.1 of \cite{A86} and the volume terms are rational numbers by condition M1. Each global orbital integral can be factored into a product of an Archimedean orbital integral and a non-Archimedean one. Since the Hecke operator $h$ is $\mathbb{Q}$-valued, the non-Archimedean orbital integrals are rational numbers by condition (M2) and Theorem 3 in \cite{Assem}. The Archimedean orbital integrals are algebraic numbers by (2). It follows that the trace of the operator $T$ is algebraic. Next, we show that, for every positive integer $m$ less than or equal to the dimension of the vector space $$\bigoplus_{\pi\in\mathcal{C}}m(\pi)\pi^{\infty,K^{\infty}},$$ the trace of $T^{m}$ is algebraic. First, we observe that $$T^{m}=\bigoplus_{\pi\in\mathcal{C}}m(\pi)\left[\pi^{\infty}(h)\right]^{m}=\bigoplus_{\pi\in\mathcal{C}}m(\pi)\pi^{\infty}(h^{*m}),$$ where $h^{*m}$ denotes the convolution of $h$ with itself $m$ times. Using the test function $f\otimes h^{*m}$, and arguing as above, we see that the trace of $T^{m}$ is algebraic. Applying the Newton-Girard identities (Proposition 28.1 in \cite{KL} and the discussion preceeding it), it follows that the coefficients of the characteristic polynomial of $T$ are algebraic numbers, and so must be the eigenvalues of $T$. Since $\pi_{0}\in\mathcal{C}$, the eigenvalues of $\pi^{\infty}_{0}(h)$ are in particular eigenvalues of $T$, and the result follows.  

\end{proof}
 
The criterion to establish that an automorphic representation is $C$-arithmetic requires strengthening condition (2) in the proposition above. We will also need the following key result, for which we are grateful to G. Wiese.

\begin{prop}\label{Gabor} Let $V$ be a finite-dimensional complex vector space of dimension $n$. Let $\mathcal{T}_{\mathbb{Q}}:=\{T_{\alpha}:V\longrightarrow V\}_{\alpha\in A}$ be a family of diagonalisable, commuting linear operators, closed under taking finite $\mathbb{Q}$-linear combinations. Assume the existence of a number field $E$ such that the characteristic polynomial $\text{char}_{\alpha}(x)$ of $T_{\alpha}$ has coefficients in $E$ for every $\alpha\in A$. Then there exists a number field $E'$ such that the eigenvalues of $T_{\alpha}$ belong to $E'$ for every $\alpha\in A$.
\end{prop}
\begin{proof} For every $\alpha\in A$, the degree of $\text{char}_{\alpha}(x)$ equals the dimension of $V$. Let $E_{\alpha}$ denote the splitting field of $T_{\alpha}$ over $E$, then $\left[E_{\alpha}:E\right]\leq n!$. Let $M$ denote the maximum of the set $$\{m\in \mathbb{N}|\text{ }m=\left[E_{\alpha}:E\right]\text{ for some }\alpha\in A\},$$
and let $T_{\alpha_{0}}\in \mathcal{T}_{\mathbb{Q}}$ be such that $\left[E_{\alpha_{0}}:E\right]=M$. By assumption, there exists a basis $\{v_{1},\dots,v_{n}\}$ of $V$ with respect to which the operators in $\mathcal{T}_{\mathbb{Q}}$ can be simultaneously diagonalised. Let $\lambda_{1},\cdots,\lambda_{n}$ be the eigenvalues of $T_{\alpha_{0}}$ corresponding to the eigenvectors $v_{1},\dots,v_{n}$. If $T_{\beta}\in\mathcal{T}_{\mathbb{Q}}$, let $\mu_{1},\dots,\mu_{n}$ be the eigenvalues of $T_{\beta}$ corresponding to the eigenvectors $v_{1},\dots,v_{n}$. For every $k\in\{1,\dots,n\}$, there are only finitely many elements $c\in E$ for which $\lambda_{k}+c\mu_{k}$ is not a primitive element of $E(\lambda_{k},\mu_{k})$. We can therefore find an element $r\in\mathbb{Q}$ such that $\lambda_{k}+r\mu_{k}$ is a primitive element of $E(\lambda_{k},\mu_{k})$ for every $k\in\{1,\dots,n\}$. Now, the operator $T_{\alpha_{0}}+rT_{\beta}$ is in $\mathcal{T}_{\mathbb{Q}}$ by assumption, its eigenvalues corresponding to the eigenvectors $v_{1},\dots,v_{n}$ are $\lambda_{1}+r\mu_{1},\dots,\lambda_{n}+r\mu_{n}$, and the splitting field over $E$ of its characteristic polynomial is $L=E(\lambda_{1}+r\mu_{1},\dots,\lambda_{n}+r\mu_{n})$. By construction, $L$ contains $E_{\alpha_{0}}$, hence $\left[L:E\right]$ is at least $M$. Moreover, $\left[L:E\right]$ is at most $M$, since $T_{\alpha_{0}}+rT_{\beta}$ is in $\mathcal{T}_{\mathbb{Q}}$. It follows that $L=E_{\alpha_{0}}$ and, since $E_{\beta}$ is contained in $L$ by construction, we have that $E_{\beta}$ is contained in $E_{\alpha_{0}}$. Since $T_{\beta}$ is an arbitrary element in $\mathcal{T}_{\mathbb{Q}}$, the result follows by setting $E'=E_{\alpha_{0}}$.

\end{proof}

\begin{cor}\label{Main} Let $G$ be a connected, semisimple, anisotropic algebraic group defined over $\mathbb{Q}$. Let $\pi_{0}=\pi_{0,\infty}\otimes \pi_{0}^{\infty}$ be an irreducible automorphic representation of $G(\mathbb{A}_{\mathbb{Q}})$. Assume the existence of a function $f\in C^{\infty}_{c}(G(\mathbb{R}))$ verifying the following conditions:
\begin{enumerate}
\item[(1)] There exists a finite family $\mathcal{C}$ of irreducible automorphic representations of $G(\mathbb{A}_{\mathbb{Q}})$ containing $\pi_{0}$ such that:
\begin{enumerate}
\item[(i)] If $\pi\in\mathcal{C}$, then $\pi^{\infty,K^{\infty}}\neq\{0\}$ and $\text{Tr}\pi_{\infty}(f)=1$.
\item[(ii)] If $\pi\notin \mathcal{C}$, then $\text{Tr}\pi_{\infty}(f)=0$.
\end{enumerate}
\item[(2')] There exists a number field $E$ such that, for every $\gamma\in \{G(\mathbb{Q})\}$, the Archimedean orbital integral $\mathcal{O}_{\gamma_{\infty}}(f)$ belongs to $E$. 
\end{enumerate}
Then there exists a number field $E'$ and a finite set of places, $S$, containing $\infty$ and all the places at which $\pi_{0}$ ramifies, such that for every $p\notin S$, and for every $h\in\mathcal{H}_{\mathbb{Q}}(G(\mathbb{Q}_{p}),K_{p})$, the eigenvalue of the operator $$\pi_{0,p}(h):\pi_{0,p}^{K_{p}}\longrightarrow \pi_{0,p}^{K_{p}}$$ is contained in $E'$; that is: the representation $\pi_{0}$ is $C$-arithmetic.
\end{cor}
\begin{proof} There exists a fnite set of places, $S$, which contains $\infty$ and such that, for $p\notin S$, the representation $\pi_{p}$ is unramified for every $\pi\in\mathcal{C}$. \\  For $p\notin S$, and for each element $h\in\mathcal{H}_{\mathbb{Q}}(G(\mathbb{Q}_{p}),K_{p})$, we can form an element in $\mathcal{H}_{\mathbb{Q}}(G(\mathbb{A}_{\mathbb{Q}}^{\infty}),K^{\infty})$, called again $h$ abusing notation, defined by tensoring $h$ with the unit of $\mathcal{H}_{\mathbb{Q}}(G(\mathbb{Q}_{p'}),K_{p'})$ for every $p'\neq p$. We thus obtain, for every $p\notin S$, and for every $\pi\in\mathcal{C}$, a family of commuting diagonalisable operators $$\{\pi^{\infty}(h):\pi^{\infty,K^{\infty}}\longrightarrow \pi^{\infty,K^{\infty}}\}_{h}$$ indexed by $\mathcal{H}_{\mathbb{Q}}(G(\mathbb{Q}_{p}),K_{p})$. For every $p\notin S$, this gives rise to a family of commuting diagonalisable operators $$\mathcal{T}_{\mathbb{Q},p}:=\{T_{h}:=\bigoplus_{\pi\in\mathcal{C}}m(\pi)\pi^{\infty,K^{\infty}}(h):\bigoplus_{\pi\in\mathcal{C}}m(\pi)\pi^{\infty,K^{\infty}}\longrightarrow \bigoplus_{\pi\in\mathcal{C}}m(\pi)\pi^{\infty,K^{\infty}}\}_{h}$$ indexed by $\mathcal{H}_{\mathbb{Q}}(G(\mathbb{Q}_{p}),K_{p})$. 
Arguing as in the proof of Proposition \ref{algtrace}, but using the stronger condition (2'), it follows that, for every $p\notin S$, the characteristic polynomial of every $T_{h}\in\mathcal{T}_{\mathbb{Q},p}$ has coefficients in $E$. 
We observe that elements belonging to different $\mathcal{T}_{\mathbb{Q},p}$'s commute. Therefore, we form the $\mathbb{Q}$-subspace $\mathcal{T_{\mathbb{Q}}}$ of the endomorphism space of $$\bigoplus_{\pi\in \mathcal{C}}\pi^{\infty,K^{\infty}}$$ generated by the families $\mathcal{T}_{\mathbb{Q},p}$. It consists of commuting diagonalisable operators. Let $T\in \mathcal{T}_{\mathbb{Q}}$ and write it as $$T=\sum_{i=1}^{r}a_{i}T_{h_{i}}$$ for some $a_{1},\dots,a_{r}\in \mathbb{Q}$ and some $h_{1},\dots,h_{r}\in \bigcup_{p\notin S}\mathcal{T}_{\mathbb{Q},p}$. We observe that $T$ is equal to the operator $$\bigoplus_{\pi\in\mathcal{C}}m(\pi)\pi^{\infty,K^{\infty}}(\tilde{h}):\bigoplus_{\pi\in\mathcal{C}}m(\pi)\pi^{\infty,K^{\infty}}\longrightarrow \bigoplus_{\pi\in\mathcal{C}}m(\pi)\pi^{\infty,K^{\infty}},$$ where $$\tilde{h}:=\sum_{i=1}^{r}a_{i}h_{i}\in\mathcal{H}_{\mathbb{Q}}(G(\mathbb{A}_{\mathbb{Q}}^{\infty}),K^{\infty}).$$ 
Arguing as in Proposition \ref{algtrace}, and by condition (2'), we conclude that the characteristic polynomial of $T$ has coefficients in $E$.
We can thus apply Proposition \ref{Gabor} to $\mathcal{T}_{\mathbb{Q}}$ and the result follows since, for every $p\notin S$ and for every $h\in\mathcal{H}_{\mathbb{Q}}(G(\mathbb{Q}_{p}),K_{p})$, the eigenvalue of the operator $\pi_{0,p}(h)$ is an eigenvalue of the operator $T_{h}\in\mathcal{T}_{\mathbb{Q},p}\subset \mathcal{T}_{\mathbb{Q}}$.      

\end{proof}

To conclude, we apply Corollary \ref{Main} to give a new proof that an automorphic representation with a regular discrete series as Archimedean component is \\ $C$-arithmetic.
These have been studied in \cite{A81}.\\

We equip the centraliser groups $G_{\gamma}(\mathbb{A}_{\mathbb{Q}})$ with Gross' measure , so that (M1) is satisfied. For every non-Archimedean $p$ at which $G(\mathbb{Q}_{p})$ is unramified, we choose local measures $\mu_{p}$ on $G({\mathbb{Q}_{p}})$ such that $\mu_{p}(G(\mathbb{Z}_{p}))=1$ and, for the remaining places, we normalise the Haar measures so that they satisfy (M2) in the way explained above.

\begin{cor}\label{Application} Let $G$ be a connected, semisimple, anisotropic algebraic group defined over $\mathbb{Q}$. Let $\pi_{0}=\pi_{0,\infty}\otimes \pi_{0}^{\infty}$ be an irreducible automorphic representation of $G(\mathbb{A}_{\mathbb{Q}})$ with $\pi_{0,\infty}$ a regular discrete series and $\pi^{\infty,K^{\infty}}_{0}\neq \{0\}$. Then $\pi_{0}$ is $C$-arithmetic.
\end{cor}
\begin{proof}  By Th\'eor\`eme 3 in \cite{CD}, if $\mu$ is an irreducible, finite-dimensional representation of $G(\mathbb{R})$, there exists a smooth compactly supported function $f_{\mu}$ such that, for every admissible $(\mathfrak{g},K_{\infty})$-module of finite length $\sigma$, we have $$\text{Tr}\sigma(f_{\mu})=\sum^{\infty}_{i=0}(-1)^{i}\text{dim}\text{H}^{i}(\mathfrak{g},K_{\infty},\sigma\otimes \mu).$$
Now, let $\pi_{0}$ be as in the statement, let $\mu$ be an irreducible, finite-dimensional, algebraic representation of $G(\mathbb{R})$ the infinitesimal character of which is the contragredient of the infinitesimal character of $\pi_{0,\infty}$. Then the $(\mathfrak{g},K_{\infty})$-cohomology of $\pi_{0,\infty}$ with respect to $\mu$ is concentrated in one degree by part (b) of Theorem 5.3 in \cite{BorelWallach} and it is $1$-dimensional: we can thus normalise $f_{\mu}$ so that $\text{Tr}\pi_{0,\infty}(f_{\mu})=1$. This will then hold for every discrete series with the same infinitesimal character as $\pi_{0,\infty}$. By the proof of Corollary 6.2 in \cite{A81}, under the regularity assumption on $\pi_{0,\infty}$, the unitary irreducible representations of $G(\mathbb{R})$ with non-vanishing $(\mathfrak{g},K_{\infty})$-cohomology with respect to $\mu$ are precisely the discrete series with the same infinitesimal character as $\pi_{0,\infty}$. It is well-known that there are only finitely many such discrete series representations, therefore (1) of Corollary \ref{Main} is fulfilled. For (2'), we observe that since $G$ is defined over $\mathbb{Q}$, the representation $\mu$ admits a model over a number field $E$. Let $\gamma\in \{G(\mathbb{Q})\}$. By Theorem 4.1 in \cite{CL}, taking into account our normalisation of $f_{\mu}$ and the choice of Haar measure, the orbital integral $\mathcal{O}_{\gamma_{\infty}}(f_{\mu})$ vanishes or we have $$\mathcal{O}_{\gamma_{\infty}}(f_{\mu})=d_{\gamma_{\infty}}\text{Tr}\mu(\gamma_{\infty})$$ where $d_{\gamma_{\infty}}\in \mathbb{Z}\backslash \{0\}$. 
Since $\mu$ is defined over $E$, the orbital integral $\mathcal{O}_{\gamma}(f_{\mu})$ belongs to $E$. We can now apply Corollary \ref{Main}, concluding the proof.

\end{proof}

\newpage

\section{Reduction to Cuspidal Representations}

We recall that by the discussion in Section 5.1 of \cite{BG}, given a connected, reductive algebraic group $G$ defined over a number field $F$, and a central extension $$1\longrightarrow \mathbb{G}_{m}\longrightarrow G'\longrightarrow G\longrightarrow 1,$$ since the map $G'(\mathbb{A}_{F})\longrightarrow G(\mathbb{A}_{F})$ is surjective, we can identify the irreducible automorphic representations of $G(\mathbb{A}_{F})$ with the irreducible automorphic representations of $G'(\mathbb{A}_{F})$ which are trivial on the image of $\mathbb{G}_{m}(\mathbb{A}_{F})$ in $G'(\mathbb{A}_{F})$. We would like to thank T. Gee for explaining that this identification follows from the results in Section 5 of \cite{LS}.

\begin{prop}\label{liftcusp} Let $G$ be a connected, reductive algebraic group defined over a number field $F$ and let $$1\longrightarrow \mathbb{G}_{m}\longrightarrow G'\longrightarrow G\longrightarrow 1$$ be a central extension. Let $\pi$ be an irreducible automorphic representation of $G(\mathbb{A}_{F})$ and let $\pi'$ denote the irreducible automorphic representation of $G'(\mathbb{A}_{F})$ obtained by lifting $\pi$ along the surjection $G'(\mathbb{A}_{F})\twoheadrightarrow G(\mathbb{A}_{F})$. Then the following two statements hold:
\begin{enumerate} 
\item[(1)] The representation $\pi$ is $C$-algebraic (resp. $C$-arithmetic, resp. $L$-algebraic, resp. $L$-arithmetic) if and only if the representation $\pi'$ is $C$-algebraic (resp. $C$-arithmetic, resp. $L$-algebraic, resp. $L$-arithmetic).
\item[(2)] If $\pi$ is cuspidal, then $\pi'$ is cuspidal.
\end{enumerate}
\end{prop}
\begin{proof} See \cite{BG}, Lemma 5.33, for part (1). Part (2) is \cite{LS}, Theorem 5.2.1. 

\end{proof}

For certain groups, the $C$ and $L$ notions that we are considering are related by a character twist. More precisely:

\begin{thm}\label{twistchi} Let $G$ be a connected, reductive algebraic group defined over a number field $F$. Assume that $G$ admits a twisting element in the sense of Definition 5.34 of \cite{BG}. Then there exists a character $\chi$ of $G(F)\backslash G(\mathbb{A}_{F})$ such that the following statements hold:
\begin{enumerate}  
\item[(1)] An irreducible automorphic representation $\pi$ of $G(\mathbb{A}_{F})$ is $C$-algebraic if and only if $\pi\otimes \chi$ is $L$-algebraic.
\item[(2)] An irreducible automorphic representation $\pi$ of $G(\mathbb{A}_{F})$ is $C$-arithmetic if and only if $\pi\otimes \chi$ is $L$-arithmetic.
\end{enumerate}
\end{thm}
\begin{proof} See Proposition 5.35 in \cite{BG} for part (1) and Proposition 5.36 in \cite{BG} for part (2).

\end{proof}

\begin{rem} A consequence of Theorem \ref{twistchi}, as explained in \cite{BG}, is that, if $G$ admits a twisting element, we can twist $L$-algebraic and $C$-algebraic representations into each other, and we can twist $C$-arithmetic and $L$-arithmetic representations into each other. Indeed, if $\pi$ is $L$-algebraic, writing it as $(\pi\otimes\chi^{-1})\otimes \chi$, we obtain that $\pi\otimes\chi^{-1}$ is $C$-algebraic by part (1) of Theorem \ref{twistchi}. Conversely, if $\pi\otimes \chi^{-1}$ is $C$-algebraic, then, twisting it by $\chi$ and applying part (1) of Theorem \ref{twistchi}, shows that $\pi$ is $L$-algebraic. The same reasoning applies for the notions of $C$-arithmetic and $L$-arithmetic automorphic representations. 
\end{rem}

If a group does not admit twisting elements, then there exists a central extension which does.

\begin{thm}\label{splitel} Let $G$ be a connected, reductive algebraic group defined over a number field $F$. Then there exists a central extension $$1\longrightarrow \mathbb{G}_{m} \longrightarrow \widetilde{G}\longrightarrow G \longrightarrow 1$$ such that $\widetilde{G}$ admits a twisting element.
\end{thm}
\begin{proof} See part (a) of Proposition 5.37 in \cite{BG}.

\end{proof}

Before proceeding, let us recall the fundamental result of R. Langlands establishing that every automorphic representation is a subquotient of the parabolic induction of a cuspidal automorphic representation. Following \cite{L}, if $P$ is a parabolic subgroup of $G$ with Levi factor $M$, and if $\sigma=\bigotimes_{v}\sigma_{v}$ is a cuspidal automorphic representation of $M(\mathbb{A}_{F})$, we call an irreducible subquotient of $\text{Ind}_{P}(\sigma)=\bigotimes_{v}\text{Ind}_{P_{v}}(\sigma_{v})$ a constituent. Recall that there exists a finite set of places, $S$, such that for $v\notin S$, the representation $\text{Ind}_{P_{v}}(\sigma_{v})$ has exactly one constituent, denoted by $\pi^{\circ}_{v}$, with non-zero $G(\mathcal{O}_{v})$-invariant vectors. 

\begin{thm}\label{autparind} Let $G$ be a connected, reductive algebraic group defined over a number field $F$. Then the following statements hold:
\begin{enumerate}
\item[(1)] If $P$ is a parabolic subgroup of $G$ with Levi factor $M$, and $\sigma=\bigotimes_{v}\sigma_{v}$ is an irreducible, cuspidal automorphic representation of $M(\mathbb{A}_{F})$, then the constituents of $\text{Ind}_{P}(\sigma)$ are the representations $\pi=\bigotimes_{v}\pi_{v}$, where $\pi_{v}$ is a constituent of $\text{Ind}_{P_{v}}(\sigma_{v})$, and, for all $v\notin S$, $\pi_{v}=\pi^{\circ}_{v}$.
\item[(2)] An irreducible representation $\pi$ of $G(\mathbb{A}_{F})$ is automorphic if and only if there exists a parabolic subgroup $P$ of $G$ with Levi factor $M$, and an irreducible, cuspidal automorphic representation $\sigma$ of $M(\mathbb{A}_{F})$ such that $\pi$ is a constituent of $\text{Ind}_{P}(\sigma)$.
\end{enumerate}
\end{thm}
\begin{proof} Part (1) is Lemma 1 in \cite{L}, part (2) is Proposition 2 in \cite{L}.
    
\end{proof}

We will make use of the following result, which we have learnt from Remark 2.10 in \cite{ST}, and for which we supply a proof.   

\begin{prop}\label{lind} Let $G$ be a connected, reductive group defined over a number field $F$, and $\pi$ be an irreducible automorphic representation of $G(\mathbb{A}_{F})$. Let $P$ be a proper parabolic subgroup of $G$ with Levi factor $M$, and $\sigma$ be an irreducible, cuspidal automorphic representation of $M(\mathbb{A}_{F})$ such that $\pi$ is a constituent of $\text{Ind}_{P}(\sigma)$. Then the following two statements hold:
\begin{enumerate}
\item[(1)] The representation $\pi$ is $L$-algebraic if and only if $\sigma$ is $L$-algebraic.
\item[(2)] If $\sigma$ is $L$-arithmetic, then $\pi$ is $L$-arithmetic.
\end{enumerate}
\end{prop}
\begin{proof} Let $\pi_{v}$ be an Archimedean component of $\pi$. Then $\pi_{v}$ is a constituent of $\text{Ind}_{P_{v}}(\sigma_{v})$, and $\sigma_{v}$ is an irreducible admissible representation of the Levi factor $M_{v}$ of $P_{v}$. Writing $M_{v}$ as $M^{\circ}_{v}A_{v}$, with $A_{v}$ denoting the split component of $M_{v}$, then we can write $\sigma_{v}$ as the tensor product of an admissible, irreducible representation of $M^{\circ}_{v}$ and a character of $A_{v}$. With this observation, it follows from Proposition 8.22 in \cite{Knapp} that $\pi_{v}$ and $\sigma_{v}$ have the same infinitesimal character. Since this is true for all Archimedean places, it follows that $\sigma$ is $L$-algebraic if and only if $\pi$ is $L$-algebraic. \\
Part (2) is a special case of Lemma 5.45 in \cite{BG}.
\end{proof}

We can now prove the main result of this section.

\begin{thm}\label{redcusp} Let $G$ be a connected, reductive group defined over a number field $F$. If every irreducible, $C$-algebraic, cuspidal automorphic representation of $G(\mathbb{A}_{F})$ is $C$-arithmetic, then every irreducible, $C$-algebraic automorphic representation of $G(\mathbb{A}_{F})$ is $C$-arithmetic, and every irreducible, $L$-algebraic automorphic representation of $G(\mathbb{A}_{F})$ is $L$-arithmetic. 
\end{thm} 
\begin{proof} We distinguish four cases.
\begin{enumerate}
\item[(I)] Let $\pi$ be an irreducible, $C$-algebraic, cuspidal automorphic representation of $G(\mathbb{A}_{F})$.  Then the result holds by assumption.
\item[(II)] Let $\pi$ be an irreducible, $L$-algebraic, cuspidal automorphic representation of $G(\mathbb{A}_{F})$. By Theorem \ref{splitel}, there exists a central extension $$1\longrightarrow \mathbb{G}_{m}\longrightarrow \widetilde{G}\longrightarrow G\longrightarrow 1$$ such that $\widetilde{G}$ admits a twisting element. Let $\tilde{\pi}$ denote the automorphic representation of $\widetilde{G}(\mathbb{A}_{F})$ obtained by lifting $\pi$ along the surjection \\ $\widetilde{G}(\mathbb{A}_{F})\twoheadrightarrow G(\mathbb{A}_{F})$.
Then, by Proposition \ref{liftcusp}, $\tilde{\pi}$ is cuspidal and $L$-algebraic. By Theorem \ref{twistchi}, we can twist $\tilde{\pi}$ to obtain a cuspidal $C$-algebraic automorphic representation of $\widetilde{G}(\mathbb{A}_{F})$, which is therefore $C$-arithmetic by assumption. By Theorem \ref{twistchi}, the representation $\tilde{\pi}$ is $L$-arithmetic, and so is $\pi$ by Proposition \ref{liftcusp}.
\item[(III)] Let $\pi$ be an irreducible, $L$-algebraic automorphic representation of $G(\mathbb{A}_{F})$ which is not cuspidal. Then, by part (2) of Theorem \ref{autparind}, $\pi$ is a constituent of $\text{Ind}_{P}(\sigma)$, where $P$ is a parabolic subgroup of $G$ with Levi factor $M$, and $\sigma$ is a cuspidal automorphic representation of $M(\mathbb{A}_{F})$. By (1) of Proposition \ref{lind}, $\sigma$ is $L$-algebraic and, arguing as in case (II), we obtain that $\sigma$ is $L$-arithmetic. By part (2) of Proposition \ref{lind}, it follows that $\pi$ is $L$-arithmetic.
\item[(IV)] Let $\pi$ be an irreducible, $C$-algebraic automorphic representation of $G(\mathbb{A}_{F})$ which is not cuspidal. By Theorem \ref{splitel}, there exists a central extension $$1\longrightarrow \mathbb{G}_{m}\longrightarrow \widetilde{G}\longrightarrow G\longrightarrow 1$$ such that $\widetilde{G}$ admits a twisting element. Let $\tilde{\pi}$ denote the automorphic representation of $\widetilde{G}(\mathbb{A}_{F})$ obtained by lifting $\pi$ along the surjection \\ $\widetilde{G}(\mathbb{A}_{F})\twoheadrightarrow G(\mathbb{A}_{F})$. By part (1) of Proposition \ref{liftcusp}, $\tilde{\pi}$ is $C$-algebraic. By Theorem \ref{twistchi}, we can twist $\tilde{\pi}$ to obtain an $L$-algebraic automorphic representation of $\widetilde{G}(\mathbb{A}_{F})$. Arguing as in case (III), we obtain that this twist of $\tilde{\pi}$ is $L$-arithmetic. By Theorem \ref{twistchi}, it follows that $\tilde{\pi}$ is $C$-arithmetic and we conclude that $\pi$ is $C$-arithmetic by part (1) of Proposition \ref{liftcusp}. 
\end{enumerate}
\end{proof}

Department of Mathematics, University of Luxembourg, Maison du \\ Nombre, 6 Avenue de la Fonte, L-4364, Esch-sur-Alzette, Luxembourg. \\

\textit{Email address:} fabio.larosa@uni.lu\\

\textbf{Data Availability:} No data has been used that cannot be retrieved from the references in the bibliography.

\end{document}